%% file: SSP.tex
\documentclass[a4paper]{article}

\usepackage{amsmath,amssymb,amsthm,hyperref,cleveref}
\usepackage{subcaption,tikz}
\usepackage{fullpage,cite}

\usetikzlibrary{arrows, fit,  calc, decorations.markings, decorations.pathmorphing, shapes}

\newtheorem{theorem}{Theorem}[section]
\newtheorem{lemma}[theorem]{Lemma}
\newtheorem{corollary}[theorem]{Corollary}
\theoremstyle{definition}
\newtheorem{definition}[theorem]{Definition}
\newtheorem{conj}[theorem]{Conjecture}
\newtheorem{question}[theorem]{Question}

\providecommand{\sqbinom}[2]{\genfrac{[}{]}{0pt}{}{#1}{#2}}

\providecommand{\meet}{\land}
\providecommand{\join}{\lor}
\providecommand{\Span}[1]{\langle #1 \rangle}
\providecommand{\covered}{\lessdot}

\providecommand{\lattice}{L}
\providecommand{\ind}[1]{\mathbf{1}_{\{#1\}}}
\providecommand{\charf}[1]{\ensuremath{\chi_{#1}}}
\providecommand{\Charf}[1]{X(#1)}
\providecommand{\FF}{\mathbb{F}}
\DeclareMathOperator{\Str}{Str}

\DeclareMathOperator{\VC}{VC}

\title{A Sauer--Shelah--Perles Lemma for Lattices}
\author{
Stijn Cambie
	\thanks{Department of Mathematics, Radboud University Nijmegen, Postbus 9010, 6500 GL Nijmegen, The Netherlands. Email: \texttt{S.Cambie@math.ru.nl}. Research supported by a Vidi Grant of the Netherlands Organization for Scientific Research (NWO), grant number $639.032.614$.} \and
Bogdan Chornomaz\thanks{Department of Mathematics, Vanderbilt University.
Email: \texttt{bogdan.chornomaz@vanderbilt.edu.}} \and
Zeev Dvir\thanks{Department of Computer Science and Department of Mathematics,
Princeton University.
Email: \texttt{zeev.dvir@gmail.com}. Research supported by NSF CAREER award DMS-1451191 and NSF grant CCF-1523816. Work on this paper was done while the author was a Von-Neumann fellow at the Institute for Advanced Study, supported by a Simons foundation fellowship.} \and Yuval Filmus\thanks{Computer Science Department, Technion --- Israel Institute of Technology.
Email: \texttt{yuvalfi@cs.technion.ac.il}. Taub Fellow --- supported by the Taub Foundations. The research was funded by ISF grant 1337/16.}
\and Shay Moran\thanks{
School of Mathematics, IAS. Email: \texttt{shaymoran1@gmail.com}.
Research supported by the National Science Foundation under agreement No.\ CCF-1412958 and by the Simons Foundations.}}

\begin{document}
\maketitle

\begin{abstract}

We study lattice-theoretical extensions of the celebrated Sauer--Shelah--Perles Lemma. 
We conjecture that a general Sauer--Shelah--Perles Lemma holds for a lattice $L$ if and only if $L$ is relatively complemented, and prove partial results towards this conjecture.




\end{abstract}

\section{Introduction} \label{sec:introduction}

Vapnik--Chervonenkis dimension~\cite{VC68,VC71}, or VC dimension for short, is a combinatorial parameter of major importance in discrete and computational geometry~\cite{HW87,CW89,KPW92}, statistical learning theory~\cite{VC71,BEHW89}, and other areas~\cite{FP94,KNR99,AMY17,HQ18}. The VC dimension of a family $F$ of binary vectors, $F \subseteq \{0,1\}^n$, is the largest cardinality of a set shattered by the family. In many of its applications, the power of this notion boils down to the Sauer--Shelah--Perles lemma~\cite{VC71,Sauer72,Shelah72}, which states that the largest cardinality of a family on $n$ points with VC dimension $d$ is 
\begin{equation}\label{eq:1}
\binom{n}{0} + \cdots + \binom{n}{d},
\end{equation}
a bound achieved by the family $\{ S : |S| \leq d \}$. This lemma follows easily from the strengthening due to Pajor~\cite{Pajor85} and Aharoni and Holzman~\cite{AharoniHolzman}, which states that every family of binary vectors shatters at least as many sets as it has:
\begin{equation}\label{eq:2}
\lvert F\rvert \leq \bigl\lvert\{X : X \text{ is shattered by } F\}\bigr\rvert
\end{equation}
 This latter version will be of our primary concern and, for the sake of compliance with existing terminology, we will call it the SSP lemma. 

VC dimension and corresponding lemmas have been extended to various settings, such as 
non-binary vectors~\cite{Steele78,KM78,Alon83,HL95}, integer vectors~\cite{Vershynin05}, Boolean matrices with forbidden configurations~\cite{Anstee85,AF10}, multivalued functions~\cite{HL95}, continuous spaces~\cite{Natarajan89}, graph powers~\cite{CBH98}, and ordered variants~\cite{ARS02}. 
In this paper, we study a generalization of the VC dimension to ranked lattices, 
and identify a rich class of lattices which satisfies the SSP lemma,  including the lattice of subspaces of a finite vector space as well as all geometric lattices (flats of matroids).
Moreover, we identify a necessary condition for a lattice to satisfy the SSP lemma:
it must be free of an induced copy of the lattice $0 <  1 <  2$.
We conjecture that this is the only obstruction; namely that a lattice satisfies the lemma if and only if it is \emph{relatively complemented (RC)}, that is, does not contain an induced copy of $0 < 1 < 2$.


\paragraph{VC dimension for ranked lattices.} The definition of shattering for binary vectors is easier to formulate in terms of sets, where we identify an $n$-bit string with its set of $1$'s and vice versa. 
A family $F$ of subsets of $\{1,\ldots,n\}$ \emph{shatters} a set $S$ if for all $T \subseteq S$, the family $F$ contains a set $A$ with $A \cap S = T$.
This definition readily generalizes to lattices (indeed, to meet-semilattices): a family $F$ of elements in a lattice $L$ \emph{shatters} an element $s \in L$ if for all $t \leq s$, the family $F$ contains an element~$a$ such that $a \meet s = t$.
If the lattice is ranked, it is natural to define the VC~dimension of a family as the maximum rank of an element it shatters. 

Given a ranked lattice $\lattice$, the family $F_d$ consisting of all elements of rank at most $d$, has VC~dimension~$d$. This family contains $\sqbinom{\lattice}{0} + \cdots + \sqbinom{\lattice}{d}$ elements, where $\sqbinom{\lattice}{d}$ is the number of elements in $\lattice$ of rank~$d$. The classical Sauer--Shelah--Perles lemma states that when $\lattice$ is the lattice of all subsets of $\{1,\ldots,n\}$, the family $F_d$ has maximum size among all families of VC~dimension $d$.

The following generalization of \Cref{eq:1} appears in an unpublished (but well circulated) manuscript of Babai and Frankl~\cite{BF92}.

\begin{theorem}[Babai and Frankl \cite{BF92}] \label{thm:main-intro}
Let $L$ be a ranked lattice of rank $r$ with nonvanishing M\"obius function, i.e.\ $\mu(x,y) \neq 0$ for all $x \leq y.$
Then for all $0 \leq d \leq r$, any family $F \subseteq L$ of VC~dimension $d$
contains at most $\sqbinom{L}{\leq \VC(F)}=\sqbinom{L}{0} + \cdots + \sqbinom{L}{d}$ elements.
Furthermore, for every $d \leq r(L)$ the inequality is tight for some $F \subseteq L$ of VC~dimension $d$.
\end{theorem}

The condition of nonvanishing M\"obius function is satisfied, for example, by a large class of lattices, namely geometric lattices, which are lattices whose elements are the flats of a finite matroid. Boolean lattices, which are the focus of the classical VC theory, also fall under this category. Another particularly compelling example of a geometric lattice is the lattice of all subspaces of $\FF_q^n$, where $\FF_q$ is the finite field of order $q$. 

Just as in the classical case, \Cref{thm:main-intro} follows from the next inequality corresponding to \Cref{eq:2}: 
\begin{theorem}[Babai and Frankl \cite{BF92}] \label{thm:main-intro-2}
If $L$ has nonvanishing M\"obius function then every family $F \subseteq L$ shatters at least $|F|$ elements.
\end{theorem}
For completeness, we provide a proof of \Cref{thm:main-intro-2} in \Cref{sec:3}.

\subsection{Towards a Characterization of Sauer-Shelah-Perles Lattices}
We will call a lattice satisfying the conclusion of \Cref{thm:main-intro-2} an SSP lattice. 
The main goal of this manuscript is to characterize such lattices.
Note that the SSP property makes sense even for lattices which are not ranked.
Having a nonvanishing M\"obius function is sufficient for a lattice to be SSP, but this condition is not necessary. Two examples are given in Figure~\ref{fig:specialexample-a} and Figure~\ref{fig:rankedRClattice_mu0}. Both lattices are SSP, but the M\"obius function vanishes on the interval consisting of the entire lattice. (The first example is simpler, but the lattice is not ranked.)




\begin{figure}[h]
	
	\begin{center}
		\begin{tikzpicture}
		\node (0) at (0.5,0) {$0$};
		\node (1) at (0,1) {$1$};
		\node (2) at (1,1) {$2$};
		\node (3) at (2,1) {$3$};
		\node (4) at (-1,1.5) {$4$};
		\node (12) at (0,2) {$12$};
		\node (13) at (1,2) {$13$};
		\node (23) at (2,2) {$23$};
		\node (1234) at (0.5,3) {$1234$};
		\draw (0) -- (4) -- (1234);
		\draw (0) -- (1) -- (12) --(1234);
		\draw (0) -- (2) -- (23) --(1234);
		\draw (0) -- (3) -- (13) --(1234);
		\draw (2) -- (12);
		\draw (1) -- (13);
		\draw (3) -- (23);
		\end{tikzpicture}
	\end{center}
	\caption{SSP lattice with vanishing M\"obius function: $\mu(0,1234) = 0$} \label{fig:specialexample-a}
\end{figure}
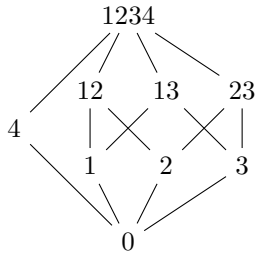

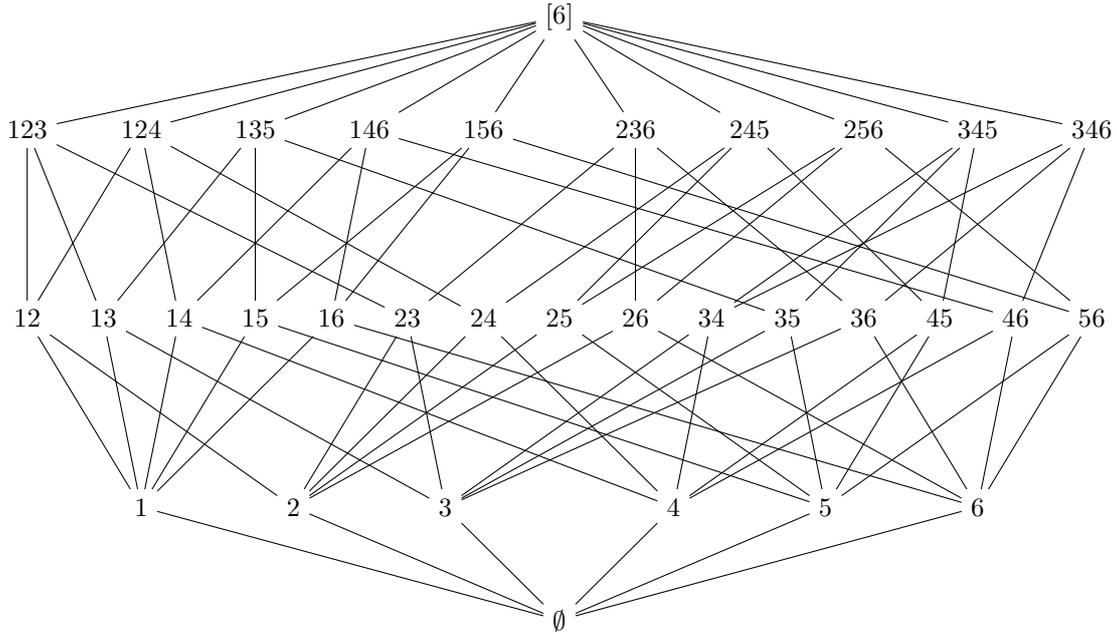
\begin{figure}[h]
	\begin{center}
		\begin{tikzpicture}
		\node (0) at (2.5,0) {$\emptyset$};
		\node (1) at (-3,1.5) {$1$};
		\node (2) at (-1,1.5) {$2$};
		\node (3) at (1,1.5) {$3$};
		\node (4) at (4,1.5) {$4$};
		\node (5) at (6,1.5) {$5$};
		\node (6) at (8,1.5) {$6$};

		\node (12) at (-4.5,4){$12$};
		\node (13) at (-3.5,4){$13$};
		\node (14) at (-2.5,4) {$14$};
		\node (15) at (-1.5,4) {$15$};
		\node (16) at (-0.5,4) {$16$};		
		\node (23) at (0.5,4) {$23$};
		\node (24) at (1.5,4) {$24$};
		\node (25) at (2.5,4) {$25$};
		\node (26) at (3.5,4) {$26$};			
		\node (34) at (4.5,4) {$34$};
		\node (35) at (5.5,4) {$35$};
		\node (36) at (6.5,4) {$36$};		
		\node (45) at (7.5,4) {$45$};
		\node (46) at (8.5,4) {$46$};
		\node (56) at (9.5,4) {$56$};	

		\node (123) at (-4.5,6.5) {$123$};
		\node (124) at (-3,6.5) {$124$};
		\node (135) at (-1.5,6.5) {$135$};
		\node (146) at (0,6.5) {$146$};	
		\node (156) at (1.5,6.5) {$156$};		
		
		\node (236) at (3.5,6.5) {$236$};
		\node (245) at (5,6.5) {$245$};
		\node (256) at (6.5,6.5) {$256$};
		\node (345) at (8,6.5) {$345$};	
		\node (346) at (9.5,6.5) {$346$};	
		
		\node (123456) at (2.5,8) {$[6]$};

		\draw (0) -- (1) -- (12) --(123)--(123456);
		\draw (0) -- (2) -- (23) --(236)--(123456);
		\draw (0) -- (3) -- (13) --(135)--(123456);
		\draw (0) -- (4) -- (34) --(345)--(123456);
		\draw (0) -- (5) -- (45) --(245)--(123456);
		\draw (0) -- (6) -- (56) --(256)--(123456);
		\draw (1) -- (13);
		\draw (1) -- (14);
		\draw (1) -- (15);
		\draw (1) -- (16);
		\draw (2) -- (12);
		\draw (2) -- (24);
		\draw (2) -- (25);
		\draw (2) -- (26);
		\draw (3) -- (23);
		\draw (3) -- (34);
		\draw (3) -- (35);
		\draw (3) -- (36);
		\draw (4) -- (14);
		\draw (4) -- (24);
		\draw (4) -- (45);
		\draw (4) -- (46);
		\draw (5) -- (15);
		\draw (5) -- (25);
		\draw (5) -- (35);
		\draw (5) -- (56);
		\draw (6) -- (16);
		\draw (6) -- (26);
		\draw (6) -- (36);
		\draw (6) -- (46);
		
		\draw (124) -- (123456);
		\draw (146) -- (123456);
		\draw (156) -- (123456);
		\draw (346) -- (123456);
		\draw (12) -- (124);
		\draw (13) -- (123);
		
		\draw (14) -- (124);
		\draw (14) -- (146);
		
		\draw (15) -- (135);
		\draw (15) -- (156);
		
		\draw (34) -- (346);
		\draw (23) -- (123);
		\draw (45) -- (345);
		\draw (56) -- (156);
		
		\draw (24) -- (124);
		\draw (24) -- (245);
		
		\draw (16) -- (146);
		\draw (16) -- (156);
		
		\draw (25) -- (245);
		\draw (25) -- (256);
		\draw (26) -- (256);
		\draw (26) -- (236);
		\draw (35) -- (345);
		\draw (35) -- (135);
		\draw (36) -- (346);
		\draw (36) -- (236);
		\draw (46) -- (346);
		\draw (46) -- (146);
		\end{tikzpicture}
	\end{center}
	
	\caption{SSP lattice with vanishing M\"obius function: $\mu(0,[6]) = 0$}
	\label{fig:rankedRClattice_mu0}
\end{figure} 

On the other hand, we can identify a large class of lattices which are not SSP:
\begin{theorem}\label{thm:notRCnotSSP}
	Let $L$ be a lattice which is not relatively complemented, i.e., contains an induced copy of $0 < 1 < 2$.
	Then there exists a family $F \subset L$ shattering strictly fewer than $\lvert F \rvert$ elements.
\end{theorem}

The lattices in Figure~\ref{fig:specialexample-a} and Figure~\ref{fig:rankedRClattice_mu0} are both relatively complemented.
Indeed, we conjecture that induced copies of $0 < 1 < 2$ are the only obstructions for the SSP property:

\begin{conj}[$SSP=RC$]\label{conj:SSP=RC}
A relatively complemented lattice is SSP.
\end{conj}

Figure~\ref{fig:counterexample-b} on page~\pageref{fig:counterexample-b} gives an example of a non-RC (and thus non-SSP) ranked lattice, which nevertheless satisfies the conclusion of Theorem~\ref{thm:main-intro}. Characterization of such lattices is thus a separate problem, which we do not address in this paper.

As partial progress towards Conjecture~\ref{conj:SSP=RC}, we prove it for lattices such that $\mu(x,y) \neq 0$ for all $x\leq y$ except possibly when $x$ is the minimal element and $y$ is the maximal element:

\begin{theorem}\label{thm:mu_vanishing_once}
	Let $L$ be an RC lattice with minimal element $0$ and maximal element $e$. If $\mu(x,y) \neq 0$ whenever $(x,y) \ne (0,e)$ then $L$ is SSP.
\end{theorem}

We also show that the SSP property is preserved under product:


\begin{theorem}\label{thm:productSSP}
	If two lattices $L$ and $K$ are SSP, then so is $L \times K$.
\end{theorem}
This theorem implies, via a structural result of Dilworth~\cite{Dil50}, that it suffices to prove Conjecture~\ref{conj:SSP=RC} for \emph{simple} relatively complemented lattices (see Dilworth's paper for a definition).

If we take a large power of any RC lattice satisfying the prerequisites of Theorem~\ref{thm:mu_vanishing_once} (such as the ones in Figure~\ref{fig:specialexample-a} and Figure~\ref{fig:rankedRClattice_mu0}) then we get an SSP lattice whose M\"obius function vanishes almost everywhere. This is a striking indication that the condition of nonvanishing M\"obius function is far from being necessary for a lattice to be SSP.


We also verify the SSP property for specific families in RC lattices:

\begin{theorem}\label{thm:n1SSP}
	If $L$ is an RC lattice and $F \subset L$ is a family for which $L \backslash \Str(F)$ contains exactly one minimal element, then $\lvert F \rvert  \leq \lvert \Str(F)\rvert$.  
\end{theorem}

\paragraph{On the proof.} 
All the results are stated for lattices, but they are true for meet-semilattices as well.
Note that a meet-semilattice with a maximal element is a lattice, so Theorem~\ref{thm:mu_vanishing_once} is still true in the meet-semilattice setting. The proofs of the other theorems also work in the meet-semilattice setting.

The original proofs of the Sauer--Shelah--Perles lemma used induction on~$n$. Alon~\cite{Alon83} and Frankl~\cite{Frankl83} gave an alternative proof using combinatorial shifting, and Frankl and Pach~\cite{FP83}, Anstee~\cite{Anstee85}, Gurvits~\cite{Gurvits97}, Smolensky~\cite{Smolensky97}, and Moran and Rashtchian~\cite{MR16} gave other proofs using the polynomial method, which the presented proof also employs.
The other proof methods --- induction and shifting --- seem to fail even for the particular case of subspace lattices.



\section{Preliminaries} \label{sec:preliminaries}

\paragraph{Posets.} A \emph{poset} is a partially ordered set. Unless mentioned otherwise, all posets we discuss are finite.
We will use $\leq$ to denote the partial order.
An \emph{antichain} is a collection of elements which are pairwise incomparable.
An element $x$ is \emph{covered} by $y$, denoted $x \covered y$, if $x < y$ and no element $z$ satisfies $x < z < y$.
We can describe a poset using its \emph{Hasse diagram}, which is a graph drawn on a plane, in which the vertices correspond to the elements, the edges to the covering relation, and lower elements are smaller.

A \emph{meet-semilattice} is a poset in which any two elements $x,y$ have an element $z \leq x,y$ such that $w \leq z$ whenever $w \leq x,y$. The element $z$ is denoted $x \meet y$, and is called the \emph{meet} of $x,y$.
The dual operation is the \emph{join} $x \join y$. A poset in which any two elements have both a meet and a join is known as a lattice.
The meet of all elements in a meet-semilattice is called the \emph{minimal element}, denoted by~0.
The join of all elements in a lattice (or join-semilattice) is called the \emph{maximal element}, denoted by~$e$.
An atom of a lattice is an element $x$ for which $0\covered x$.

A meet-semilattice is \emph{ranked} if each element $x$ is associated with a non-negative integer rank $r(x)$, subject to the following two constraints (which completely specify the rank): $r(0) = 0$, and $r(y) = r(x) + 1$ if $x \covered y$. Not every meet-semilattice can be ranked.
The rank of a meet-semilattice is the maximum rank of an element.
We denote the number of elements of rank $d$ in a poset $\lattice$ by $\sqbinom{\lattice}{d}$, and the number of elements of rank at most $d$ by $\sqbinom{\lattice}{\leq d}$.

The standard example of a lattice is the \emph{Boolean lattice} of all subsets of $\{1,\ldots,n\}$ ordered by inclusion. The meet of two elements is their intersection, and the join of two elements is their union. The rank of a subset is its cardinality.

The product $L\times K$ of two lattices $L,K$ is a lattice whose elements are the elements of a Cartesian product of $L$ and $K$, with the order relation $(\ell_1,k_1) \leq (\ell_2,k_2)$ iff $\ell_1 \leq \ell_2$ and $k_1 \leq k_2$. 

\paragraph{M\"obius function.} The M\"obius function of a finite poset is a function $\mu(x,y)$ defined for any two elements $x \leq y$ in the following way: $\mu(x,x) = 1$, and for $x < y$,
\[
 \mu(x,y) = - \sum_{z\colon x \leq z < y} \mu(x,z).
\]
For example, on the Boolean lattice the M\"obius function is $\mu(x,y) = (-1)^{|y \setminus x|}$, and on the integer lattice (the divisors of $n$ ordered by divisibility) the M\"obius function is $\mu(x,y) = \mu(y/x)$, where $\mu(\cdot)$ is the number-theoretic M\"obius function.

The M\"obius function is important due to the two \emph{M\"obius inversion formulas}:

\begin{lemma}[M\"obius inversion] \label{lem:mobius-inversion}
If $f,g$ are two real-valued functions on a poset then
\[
 f(x) = \sum_{y \geq x} g(y) \text{ for all } x \Longleftrightarrow
 g(x) = \sum_{y \geq x} \mu(x,y) f(y) \text{ for all } x.
\]
and
\[
 f(y) = \sum_{x \leq y} g(x) \text{ for all } y \Longleftrightarrow
 g(y) = \sum_{x \leq y} \mu(x,y) f(x) \text{ for all } y.
\]
\end{lemma}

We say that a poset has \emph{nonvanishing M\"obius function} if $\mu(x,y) \neq 0$ for all $x \leq y$ in the poset. For example, the Boolean lattice has nonvanishing M\"obius function, and the integer lattice has nonvanishing M\"obius function if and only if $n$ is squarefree.

\paragraph{Relatively complemented lattices.}

In a lattice $L$ which has minimal and maximal elements, the \emph{complement} of an element $y$ is an element $z$ for which $z \meet y=0$ and $z \join y=e$.
A complemented lattice is a lattice with minimal and maximal elements in which every element has a complement.
A lattice or meet-semilattice is called \emph{relatively complemented} (RC) if every interval $[x,y]=\{z \mid x \le z \le y \}$, considered as a sublattice, is complemented. In particular, an RC lattice is complemented.
We will mostly use the following equivalent characterization by Bj\"orner~\cite{B81}: 
\begin{lemma}[Bj\"orner]
	A finite lattice is RC if and only if it does not contain a $3$-element interval, i.e.\ there are no two elements $x < y$ such that there is a unique $z$ satisfying $x < z < y$.
\end{lemma}	

We will use another simple property of RC lattices due to Bj\"orner:
\begin{lemma}\label{lem:atomic}
	Finite RC lattices are atomic, that is, $e$ is the join of all atoms. Equivalently, for any $x<e$ there is an atom $a$ such that $a\not\leq x$.
\end{lemma}

Deeper structural results on RC lattices can be found in~\cite{Dil50}.

\paragraph{Matroids and geometric lattices.}
A \emph{matroid} over a finite set $U$ is a finite non-empty collection of subsets of $U$ called \emph{indepedent sets}, satisfying the following two axioms: if a set is independent, then so are all its subsets; and if $A,B$ are independent and $|A|>|B|$, then there exists an element $x \in A \setminus B$ such that $B \cup \{x\}$ is also independent.

The \emph{rank} of a subset $S \subseteq U$ is the maximum cardinality of a subset of $S$ which is independent. The rank of a matroid is the rank of $U$. A \emph{flat} is a subset of $U$ whose supersets all have higher rank.

Given a matroid, we can construct a poset whose elements are all flats of the matroid, ordered by inclusion. This poset forms a ranked lattice, and a lattice formed in this way is called a \emph{geometric lattice}. The rank of an element in the lattice is the rank of the corresponding flat in the matroid. Weisner's theorem implies that geometric lattices have nonvanishing M\"obius functions:

\begin{theorem}[Weisner] \label{thm:weisner}
The M\"obius function of a geometric lattice satisfies $(-1)^{r(y)-r(x)} \mu(x,y) > 0$ for all $x \leq y$.
\end{theorem}

\noindent For a proof, see~\cite[Corollary 16.3]{Godsil18}.

The collection of all subsets of $\{1,\ldots,n\}$ forms a matroid of rank $n$ whose flats are all subsets of $\{1,\ldots,n\}$. The corresponding geometric lattice is the Boolean lattice described above. A more interesting example of a matroid is the collection of all subsets of $\FF_q^n$ which are linearly independent, which forms a matroid of rank $n$ whose flats are all subspaces of $\FF_q^n$. The corresponding geometric lattice is called  the \emph{subspace lattice} of $\FF_q^n$.

\section{VC theory for lattices}\label{sec:3}

\subsection{Definitions}

In order to develop VC theory for lattices we need to define two concepts: shattering and VC dimension. We start with the more basic concept, shattering:

\begin{definition}[Shattering] \label{def:shattering}
Let $\lattice$ be a meet-semilattice. A set $F \subseteq \lattice$ \emph{shatters} an element $y \in \lattice$ if for all $x \leq y$ there exists an element $z \in F$ such that $z \meet y = x$.

We denote the set of all elements shattered by $F$ by $\Str(F)$.
\end{definition}
We comment that the definition can be extended further to general posets: in this case, the condition~$z \meet y = x$ should be understood as follows: $z \meet y$ exists, and equals $x$.

A basic property of shattering is that it is hereditary:

\begin{lemma} \label{lem:shattering-hereditary}
Let $\lattice$ be a meet-semilattice. If a set $F \subseteq \lattice$ shatters an element $z \in \lattice$ and $y \leq z$, then $F$ also shatters $y$. In other words, $\Str(F)$ is downward-closed.
\end{lemma}
\begin{proof}
Let $x \leq y$. Since $F$ shatters $z$ and $x \leq z$, there exists an element $w \in F$ satisfying $w \meet z = x$. Since $y \leq z$, the same element satisfies $w \meet y = w \meet (y \meet z) = (w \meet z) \meet y = x \meet y = x$.
\end{proof}

Having defined shattering, the definition of VC~dimension is obvious:

\begin{definition}
Let $\lattice$ be a ranked meet-semilattice. The \emph{VC dimension} of a non-empty set $F \subseteq \lattice$, denoted $\VC(F)$, is the maximum rank of an element shattered by $F$.
\end{definition}

These definitions specialize to the classical ones in the case of the Boolean lattice.

\subsection{Proof of \Cref{thm:main-intro-2}}\label{subsec:Proof_thm:main-intro-2}

%

Let $\FF$ be an arbitrary field of characteristic zero. We will prove \Cref{thm:main-intro-2} by giving a spanning set of size $|\Str(F)|$ for the $|F|$-dimensional vector space $\FF[F]$ of $\FF$-valued functions on $F$. \Cref{thm:main-intro-2} then follows, since the cardinality of any spanning is at least the dimension. The spanning set we construct will consist of functions of the form given by the following definition:

\begin{definition} \label{def:charf}
For $x \in \lattice$, the function $\charf{x}\colon \lattice \to \FF$ is given by
\[
 \charf{x}(y) = \ind{y \geq x},
\]
that is, $\charf{x}(y) = 1$ if $y \geq x$, and otherwise $\charf{x}(y) = 0$.

For a set $G \subseteq \lattice$,
\[
 \Charf{G} = \{ \charf{x} : x \in G \}.
\]
\end{definition}

In the case of the Boolean lattice, we can think of the elements of the lattice as encoded by sets $S \subseteq \{1,\ldots,n\}$ as well as by Boolean variables $x_1,\ldots,x_n$. The reader can verify that
\[
 \charf{S} = \prod_{i \in S} x_i.
\]
\Cref{def:charf} extends this idea to general posets.

We will show that $\FF[F]$ is spanned by $\Charf{\Str(F)}$. The first step is showing that $\Charf{\lattice}$ is a basis for~$\FF[\lattice]$, which for the Boolean lattice just states that every function on $\{0,1\}^n$ can be expressed uniquely as a multilinear polynomial:

\begin{lemma} \label{lem:charf-basis}
The set $\Charf{\lattice}$ is a basis for $\FF[\lattice]$.
\end{lemma}
\begin{proof}
Since $|\Charf{\lattice}| = |\lattice| = \dim \FF[\lattice]$, it suffices to show that $\Charf{\lattice}$ is linearly independent. Consider any linear dependency of the form $\ell := \sum_x c_x \charf{x} = 0$. We will show that $c_x = 0$ for all $x \in \lattice$, and so~$\Charf{\lattice}$ is linearly independent.

Arrange the elements of $\lattice$ in an order $x_1,\ldots,x_N$ such that $x_i < x_j$ implies $i < j$. We prove that~$c_{x_i} = 0$ by induction on $i$. Suppose that $c_{x_j} = 0$ for all $j < i$. Then in particular, $c_{x_j}=0$ for all~$x_j < x_i$, and therefore
\[
 0 = \ell(x_i) = \sum_j c_{x_j} \charf{x_j}(x_i) = \sum_{j\colon x_j \leq x_i} c_{x_j} = c_{x_i}. \qedhere
\]
\end{proof}

The crucial step of the proof of \Cref{thm:main-intro-2} is an application of (generalized) inclusion-exclusion, which shows that if $F$ does not shatter $z$ then $\charf{z}|_F$ can be expressed as a linear combination of $\charf{w}|_F$ for~$w < z$. In the case of the Boolean lattice, the argument is as follows. Suppose that $F$ does not shatter $S$, say $A \cap S \neq T$ for all $A \in F$. Then all elements of $F$ satisfy
\[
 \prod_{i \in T} x_i \prod_{j \in S \setminus T} (1 - x_j) = 0,
\]
which implies that over $F$,
\[
 \prod_{i \in S} x_i = \sum_{R \subsetneq S \setminus T} (-1)^{|S \setminus (T \cup R)|+1} \prod_{j \in T \cup R} x_j.
\]
The argument for general posets is very similar, and uses M\"obius inversion:

\begin{lemma} \label{lem:elimination}
Suppose that $z \notin \Str(F)$. There exist coefficients $\gamma_y$ such that for all $p \in F$,
\[
 \charf{z}(p) = \sum_{y < z} \gamma_y \charf{y}(p).
\]
\end{lemma}
\begin{proof}
For an element $p \in F$, define the following two functions:
\begin{align*}
f_p(x) &= \ind{x\leq p \meet z}, &
g_p(y) &= \ind{y = p \meet z}.
\end{align*}
Clearly $f_p(x) = \sum_{y \geq x} g_p(y)$, and so \Cref{lem:mobius-inversion} shows that $g_p(x) = \sum_{y \geq x} \mu(x,y) f_p(y)$. Since $f_p(y) = 0$ unless $y \leq z$, we can restrict the sum to the range $x \leq y \leq z$. When $y \leq z$, the condition $y \leq p \meet z$ is equivalent to the condition $y \leq p$, and so we conclude that
\[
 g_p(x) = \sum_{x \leq y \leq z} \mu(x,y) f_p(y) = \sum_{x \leq y \leq z} \mu(x,y) \charf{y}(p).
\]

Since $z \notin \Str(F)$, there exists an element $x \leq z$ such that $p \land z \neq x$ for all $p \in F$. In other words,~$g_p(x) = 0$ for all $p \in F$. Therefore all $p \in F$ satisfy
\[
 \charf{z}(p) = -\sum_{x \leq y < z} \frac{\mu(x,y)}{\mu(x,z)} \charf{y}(p),
\]
using the nonvanishing of the M\"obius function.
\end{proof}

We can now complete the proof, employing exactly the same argument used for the Boolean lattice.

\newtheorem*{thm:main-intro-2}{Theorem \ref{thm:main-intro-2}}
\begin{thm:main-intro-2}
	If $L$ is a lattice in which $\mu(x,y) \neq 0$ for all $x \leq y$ then every family $F \subseteq L$ shatters at least $|F|$ elements.
\end{thm:main-intro-2}

\begin{proof}
\Cref{lem:charf-basis} shows that $\Charf{\lattice}$ is a basis for $\FF[\lattice]$, and so the functions $\charf{x}$, restricted to the domain $F$, span $\FF[F]$. We will show that every function in $\FF[F]$ can be expressed as a linear combination of functions in $\Charf{\Str(F)}$.

Consider any function $f \in \FF[F]$. Since $\Charf{\lattice}$ spans $\FF[f]$, there exist coefficients $c_x$ such that $f = \sum_x c_x \charf{x}$. Define the potential function
\[
 \Phi(\vec{c}) = \sum_{\substack{x \notin \Str(F)\colon \\ c_x \neq 0}} N^{r(x)},
\]
where $N = |\lattice| + 1$, and choose a representation which minimizes $\Phi(\vec{c})$. If $\Phi(\vec{c}) > 0$ then choose $z \notin \Str(F)$ satisfying $c_z \neq 0$ of maximal rank. \Cref{lem:elimination} shows that
\[
 f = \sum_{x \neq z} c_x \charf{x} + \sum_{y < z} \gamma_y c_z \charf{y}.
\]
The corresponding coefficient vector $\vec{d}$ satisfies $\Phi(\vec{d}) < \Phi(\vec{c})$, contradicting the choice of $\vec{c}$. We conclude that $\Phi(\vec{c}) = 0$, and so $f$ is a linear combination of functions in $\Charf{\Str(F)}$.

Concluding, we have shown that $\Charf{\Str(F)}$ spans $\FF[F]$. Hence $|\Str(F)| = |\Charf{\Str(F)}| \geq \dim \FF[F]=|F|$.
\end{proof}

\subsection{Some corollaries}

\Cref{thm:main-intro-2} immediately implies \Cref{thm:main-intro}.

\newtheorem*{thm:main-intro}{Theorem \ref{thm:main-intro}}
\begin{thm:main-intro}
	Let $L$ be a ranked lattice of rank $r$ with nonvanishing M\"obius function, i.e.\ $\mu(x,y) \neq 0$ for all $x \leq y$.
	Then for all $0 \leq d \leq r$, any family $F \subseteq L$ of VC~dimension $d$
	contains at most $\sqbinom{L}{\leq \VC(F)}=\sqbinom{L}{0} + \cdots + \sqbinom{L}{d}$ elements.
	Furthermore, for every $d \leq r(L)$ the inequality is tight for some $F \subseteq L$ of VC dimension $d$. 
\end{thm:main-intro}

\begin{proof}
Suppose that $\VC(F) = d$. If $|F| > \sqbinom{L}{\leq d}$ then, according to \Cref{thm:main-intro-2}, also $|\Str(F)| > \sqbinom{L}{\leq d}$. However, this implies that $\Str(F)$ must contain a set of rank larger than $d$, contradicting the assumption~$\VC(F) = d$. This proves the inequality.

To show that the inequality is tight for all $d \leq r(L)$, consider the set $F_d = \{ x : r(x) \leq d \}$. This is a set containing $\sqbinom{L}{\leq d}$ elements which shatters all elements of rank $d$ but no element of rank $d+1$, and so satisfies $\VC(F_d) = d$.
\end{proof}

We can generalize \Cref{thm:main-intro} to arbitrary antichains to obtain a further corollary.

\begin{corollary} \label{cor:main-antichain}
Let $\lattice$ be a ranked lattice with nonvanishing M\"obius function and let $A \subseteq \lattice$ be a maximal antichain.
If $F \subseteq \lattice$ does not shatter any element of $A$ then
\[
 |F| \leq |F_A|, \text{ where } F_A = \{ x \in \lattice : x < y \text{ for some } y \in A \}.
\]
Furthermore, $F_A$ does not shatter any element of $A$.
\end{corollary}
\Cref{thm:main-intro} is the special case of \Cref{cor:main-antichain} in which $A$ consists of all elements of rank $\VC(F)+1$.
\begin{proof}[Proof of \Cref{cor:main-antichain}]
Let us start by showing that $|F| \leq |F_A|$. If $|F| > |F_A|$ then, according to \Cref{thm:main-intro-2}, also $|\Str(F)| > |F_A|$. Therefore $F$ shatters some element $x$ such that $x \not< y$ for all $y \in A$. Since $A$ is a maximal antichain, either $x \in A$ or $x \geq y$ for some $y \in A$. In both cases $F$ shatters some element in $A$ (in the second case, according to \Cref{lem:shattering-hereditary}).

Next, let us show that $F_A$ does not shatter any element of $A$. Suppose that $F_A$ shatters some element~$a \in A$. Then some $x \in F_A$ satisfies $x \meet a = a$, that is, $x \geq a$. Since $x \in F_A$, we know that $x < y$ for some~$y \in A$. Put together, this implies that $a < y$, contradicting the fact that $A$ is an antichain.
\end{proof}

\medskip

A final corollary is a \emph{dichotomy theorem}, a direct consequence of the Sauer--Shelah--Perles lemma which is the source of many of its applications. Before describing our generalized dichotomy theorem, let us briefly describe the classical one. Let $F \subseteq \{0,1\}^X$, where $X$ is infinite. For every finite $I \subseteq X$, we can consider the projection of $F$ to the coordinates of $I$, denoted $F|_I$. 
The \emph{growth function} of $F$ is
\[
 \Pi_F(n) = \max_{\substack{I \subset X \\ |I| = n}} \bigl| F|_I \bigr|.
\]
The Sauer--Shelah--Perles lemma immediately implies the following polynomial versus exponential dichotomy for the growth function:
\begin{itemize}
\item Either $\VC(F) = \infty$, in which case $\Pi_F(n) = 2^n$;
\item or $\VC(F) = d < \infty$, in which case $\Pi_F(n) \leq \binom{n}{\leq d} \leq 2n^d$.
\end{itemize}
For example, it implies that there is no $F$
for which $\pi_F(n) = \Theta(2^{\log^2 n})$.

We can extend this result to vector spaces (we leave extensions to more general domains to the reader). Let $\FF_q$ be a finite field, let $X$ be an infinite set, let $\mathcal{V}$ denote the linear space of all functions~$v\colon X\to\FF_q$ with a finite support (i.e.\ $v(x)=0$ for all but finitely many $x\in X$), and let $\mathcal{L}$ denote the (infinite) lattice of all finite dimensional subspaces of $\mathcal{V}$. Let $F\subseteq \mathcal{L}$ be a family of subspaces. For every $I\in\mathcal{L}$, we can consider the projection $F|_I = \{ V \cap I : V \in F \}$. The growth function of $F$ is defined as in the classical case, with dimension replacing cardinality:
\[
 \Pi_F(n) = \max_{\substack{I \in\mathcal{L} \\ \dim(I) = n}} \bigl| F|_I \bigr|.
\]

\Cref{thm:main-intro} immediately implies a dichotomy as in the classical case. In order to understand the resulting orders of growth, we need to be able to estimate $\sqbinom{\lattice}{d}$ for subspace lattices~$\lattice$.

\begin{lemma} \label{lem:subspace-size}
Let $\lattice = \FF_q^n$. For all $d \leq n$,
\[
 q^{d(n-d)} \leq \sqbinom{\lattice}{\leq d} \leq 2n^d q^{dn}.
\]
In particular, $|\lattice| \geq q^{(n^2-1)/4}$.
\end{lemma}
\begin{proof}
The number of elements of $\lattice$ of rank $d$ is the $q$-binomial coefficient $\sqbinom{n}{d}_q$. There are many formulas for $\sqbinom{n}{d}_q$. The one we use is
\[
 \sqbinom{n}{d}_q = \sum_{\substack{A \subseteq \{1,\ldots,n\} \\ |A|=d}} q^{\sum_{i \in A} i - d(d+1)/2}.
\]
Calculation shows that the summand with highest exponent, corresponding to $A = \{n-d+1,\ldots,n\}$, has exponent $d(n-d)$. Therefore
\[
 q^{d(n-d)} \leq \sqbinom{n}{d}_q \leq \binom{n}{d} q^{d(n-d)} \leq n^d q^{dn}.
\]
This implies that
\[
 \sqbinom{\lattice}{\leq d} \leq \sum_{e=0}^d n^e q^{en} \leq n^d q^{dn} \sum_{e=0}^d (nq^n)^{-e}.
\]
We can assume that $n \geq 1$, and so $nq^n \geq 2$, implying that $\sum_{e=0}^\infty (nq^n)^{-e} \leq 2$. This proves the main inequalities. The lower bound on $|\lattice|$ follows from taking $d = \lfloor n/2 \rfloor$.
\end{proof}

Combining the lemma with \Cref{thm:main-intro} specialized to the subspace lattice, we immediately obtain the following dichotomy theorem:

\begin{theorem} \label{thm:dichotomy-vs}
For every family $F \subseteq \mathcal{L}$, exactly one of the following holds:
\begin{itemize}
\item Either $\VC(F) = \infty$, in which case $\Pi_F(n) \geq q^{(n^2-1)/4}$;
\item or $\VC(F) = d < \infty$, in which case $\Pi_F(n) \leq 2n^dq^{dn}$.
\end{itemize}
\end{theorem}

\section{Partial results towards the SSP=RC conjecture}


\subsection{Proof of Theorem~\ref{thm:notRCnotSSP}}

We start by showing that an SSP lattice must be RC.

\newtheorem*{thm:notRCnotSSP}{Theorem \ref{thm:notRCnotSSP}}
\begin{thm:notRCnotSSP}
	Let $L$ be a lattice which is not relatively complemented.
	Then there exists a family $F \subset L$ shattering strictly fewer than $\lvert F \rvert$ elements.
\end{thm:notRCnotSSP}
\begin{proof}
	Since $L$ is not relatively complemented, there exist two elements $x,y$ in $L$ such that there is a unique element $z$ satisfying $x < z < y$.
	Let $F= \{ w \mid w \le y\} \backslash \{x\}$.
	It is easy to check that $\Str(F) \subseteq \{w \mid w \le y\}$. On the other hand, $y,z$ are not shattered since $x \notin F$. Therefore
	 $\lvert \Str (F) \rvert \le \lvert F \rvert-1$.
\end{proof}

Figure~\ref{fig:nonRClattices} shows two examples of ranked lattices that are not relatively complemented, only one of the two satisfying the conclusion of Theorem~\ref{thm:main-intro}.

The first example $L_1$ is the path lattice $0 < 1 < 2$, appearing in Figure~\ref{fig:counterexample-path}. The family $F=\{1,2\}$ shatters only $0$ and hence has VC dimension $0$, while $\lvert F \rvert > \sqbinom{L_1}{0}$.

The second example $L_2$, appearing in Figure~\ref{fig:counterexample-b}, is more complex. One can check that $\lvert \Str (F) \rvert \ge \lvert F \rvert$ unless $F = L_2 \setminus 4,L_2 \setminus 5$. Both of these families shatter $12$, and so have VC~dimension~$2$. Since $\sqbinom{L_2}{\leq 2} = |L_2|-1$, the lattice satisfies the conclusion of Theorem~\ref{thm:main-intro}.

The lattice $L_2$ thus separates the SSP condition from the weaker condition given by Theorem~\ref{thm:main-intro}. It seems hard to characterize the lattices that are ``weakly SSP'', that is, satisfy $|\Str(F)| \leq \binom{L}{\leq \VC(F)}$ for all $F \subseteq L$.


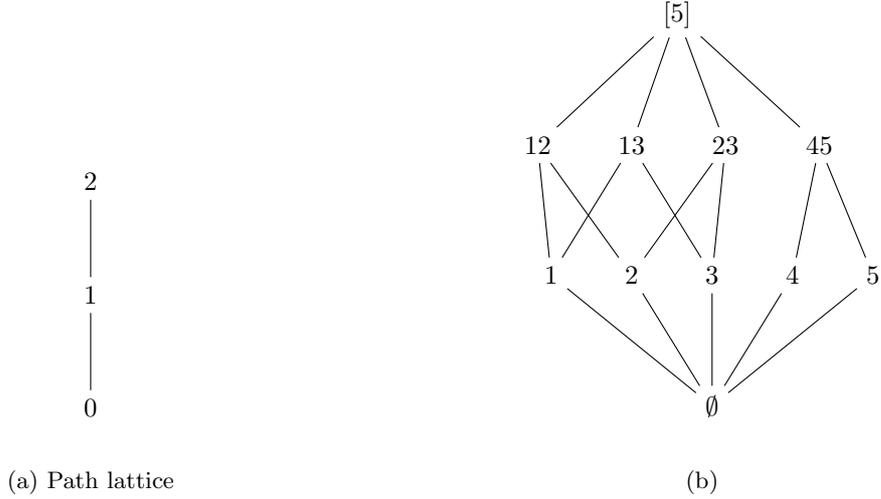
\begin{figure}[h]
\begin{minipage}[b]{.5\linewidth}
\begin{center}
\begin{tikzpicture}
\node (zero) at (0,0) {$0$};
\node (one) at (0,1.5) {$1$};
\node (two) at (0,3) {$2$};
\draw (zero) -- (one) -- (two);
\end{tikzpicture}
\end{center}
\subcaption{Path lattice} \label{fig:counterexample-path}
\end{minipage}
\begin{minipage}[b]{.5\linewidth}
\begin{center}
	\begin{tikzpicture}[>=latex,line join=bevel,]
	\node (node_9) at (113.5bp,104.5bp) [draw,draw=none] {$45$};
	\node (node_8) at (78.5bp,104.5bp) [draw,draw=none] {$23$};
	\node (node_7) at (43.5bp,104.5bp) [draw,draw=none] {$13$};
	\node (node_6) at (8.5bp,104.5bp) [draw,draw=none] {$12$};
	\node (node_4) at (133.5bp,55.5bp) [draw,draw=none] {$5$};
	\node (node_5) at (103.5bp,55.5bp) [draw,draw=none] {$4$};
	\node (node_3) at (73.5bp,55.5bp) [draw,draw=none] {$3$};
	\node (node_2) at (43.5bp,55.5bp) [draw,draw=none] {$2$};
	\node (node_1) at (13.5bp,55.5bp) [draw,draw=none] {$1$};
	\node (node_0) at (73.5bp,6.5bp) [draw,draw=none] {$\emptyset$};
	\node (node_10) at (60.5bp,153.5bp) [draw,draw=none] {$[5]$};
	\draw  (node_0) ..controls (73.5bp,19.822bp) and (73.5bp,29.898bp)  .. (node_3);
	\draw  (node_0) ..controls (65.119bp,20.188bp) and (58.475bp,31.042bp)  .. (node_2);
	\draw  (node_8) ..controls (73.552bp,117.97bp) and (69.738bp,128.35bp)  .. (node_10);
	\draw  (node_6) ..controls (23.492bp,118.63bp) and (36.032bp,130.44bp)  .. (node_10);
	\draw  (node_3) ..controls (74.859bp,68.822bp) and (75.888bp,78.898bp)  .. (node_8);
	\draw  (node_5) ..controls (106.23bp,68.896bp) and (108.32bp,79.125bp)  .. (node_9);
	\draw  (node_0) ..controls (81.881bp,20.188bp) and (88.525bp,31.042bp)  .. (node_5);
	\draw  (node_1) ..controls (21.881bp,69.188bp) and (28.525bp,80.042bp)  .. (node_7);
	\draw  (node_0) ..controls (89.262bp,19.373bp) and (106.84bp,33.731bp)  .. (node_4);
	\draw  (node_2) ..controls (53.382bp,69.334bp) and (61.361bp,80.505bp)  .. (node_8);
	\draw  (node_7) ..controls (48.173bp,117.97bp) and (51.775bp,128.35bp)  .. (node_10);
	\draw  (node_2) ..controls (33.618bp,69.334bp) and (25.639bp,80.505bp)  .. (node_6);
	\draw  (node_0) ..controls (57.738bp,19.373bp) and (40.156bp,33.731bp)  .. (node_1);
	\draw  (node_9) ..controls (98.22bp,118.63bp) and (85.439bp,130.44bp)  .. (node_10);
	\draw  (node_1) ..controls (12.141bp,68.822bp) and (11.112bp,78.898bp)  .. (node_6);
	\draw  (node_4) ..controls (127.97bp,69.042bp) and (123.67bp,79.582bp)  .. (node_9);
	\draw  (node_3) ..controls (65.119bp,69.188bp) and (58.475bp,80.042bp)  .. (node_7);
	\end{tikzpicture}
\end{center}
\subcaption{} \label{fig:counterexample-b}
\end{minipage}
\caption{Lattices that are not RC}
\label{fig:nonRClattices}
\end{figure}

\subsection{Proof of Theorem~\ref{thm:mu_vanishing_once}}

In this subsection we prove Theorem~\ref{thm:mu_vanishing_once}, which concerns RC lattices for which $\mu(x,y) \neq 0$ holds whenever $(x,y) \neq (0,e)$ (recall that $0$ is the minimal element and $e$ is the maximal element). Such lattices do exist, for example the two lattices given in the introduction (Figure~\ref{fig:specialexample-a} and Figure~\ref{fig:rankedRClattice_mu0}).



%


The crucial tool for the proof will be the following analog of Lemma~\ref{lem:elimination} for this case.

\begin{lemma}\label{lem:adapted3.7}
	Let $F \subset L$ be a family different from $L$ and $L \backslash 0$.
	Let $z \in L$ be such that $z \not \in \Str(F)$.
	There exist coefficients $\lambda_y$ such that for all $p \in F$ we have
	\[\chi_z(p)=\sum_{y<z} \lambda_y \chi_y(p).\]
\end{lemma}

\begin{proof}
	If $\mu$ is nonvanishing or $z \neq e$ then $\mu(x,z) \neq 0$ for all $x \le z$, and so the original proof of Lemma~\ref{lem:elimination} still applies.
	It therefore suffices to handle the case $z=e$ and $\mu(0,e)=0$.
	
	For every $p \in L$, define a function $v_p$ on $L\setminus e$ by $v_p(y) = \chi_y(p)$ (in other words, if we think of the functions $\chi_y$ for $y < e$ as the rows of a matrix, then the functions $v_p$ are its columns); so $v_p(y)$ indicates the condition ``$y \leq p$''. We will show that the space of linear dependencies of $\{ v_p : p \in L \}$ is one-dimensional. More explicitly, we show that if $\sum_{a \in L} c_a v_a = 0$ then $c_a = \mu(a,e) c_e$ for all $a \in L$.
	
	The proof is by backwards induction. The base case, $a = e$, is trivial. Now suppose that $c_b = \mu(b,e) c_e$ for all $b > a$. Then
\[
 0 = \sum_{b \in L} c_b v_b(a) = c_a + \sum_{b>a} \mu(b,e) c_e = c_a - \mu(a,e)c_e,
\]
    by definition of the M\"obius function.
    
    As $F \neq L,L\setminus 0$, there must be some element $a \neq 0$ missing from $F$. We claim that the functions $\{v_p : p \in F\}$ are linearly independent. Indeed, any linear dependency $\sum_{p \in F} c_p v_p = 0$ lifts to a linear dependency of $\{v_p : p \in L\}$, with $c_p = 0$ for $p \notin F$. In this linear dependency, $c_a = 0$. Since $a \neq 0$, we have $\mu(a,e) \neq 0$, and so $c_e = c_a/\mu(a,e) = 0$, which implies that the linear dependency is trivial.
    
    Since the functions $\{v_p : p \in F\}$ are linearly independent, the $(L \setminus 0) \times F$ matrix whose columns are $v_p$ has full rank $|F|$, and so its rows (which are just the functions $\chi_y|_F$ for $y < e$) span $\FF[F]$. In particular, some linear combination of the rows equals~$\chi_e|_F$.
\end{proof}

Theorem \ref{thm:mu_vanishing_once} follows by a striaghtforward adaptation of the proof of Theorem~\ref{thm:main-intro-2}.

\newtheorem*{thm:mu_vanishing_once}{Theorem \ref{thm:mu_vanishing_once}}
\begin{thm:mu_vanishing_once}
	Let $L$ be an RC lattice with minimal element $0$ and maximal element $e$. If $\mu(x,y) \neq 0$ whenever $(x,y) \neq (0,e)$ then every family $F \subseteq L$ shatters at least $|F|$ elements.
\end{thm:mu_vanishing_once}

\begin{proof}
	When $F$ is different from $L$ and $L \backslash 0$, the conclusion holds exactly as in the proof of \Cref{thm:main-intro-2} in Subsection~\ref{subsec:Proof_thm:main-intro-2}, by using Lemma~\ref{lem:adapted3.7} instead of Lemma~\ref{lem:elimination}.
	When $F=L$, we obviously have $\lvert F \rvert \le \lvert \Str (F) \rvert $  as $F=\Str(F)$.
	We argue that when $F=L \backslash 0$, then $\Str(F)=L \backslash e$.
	Indeed, take $y \in L \backslash e$.
	Any $s \le y$ different from $0$ belongs to $F$, and $y \meet s=s$.
	By Lemma~\ref{lem:atomic}, there is an atom $a \in F$ such that $a\not\leq y$. Thus $y \meet a = 0$, and so $y$ is shattered.	
\end{proof}

\subsection{Proof of Theorem~\ref{thm:productSSP}}

In this subsection we prove that the SSP property is closed under taking products.

\newtheorem*{thm:productSSP}{Theorem \ref{thm:productSSP}}
\begin{thm:productSSP}
		If two lattices $L$ and $K$ are SSP, then $L \times K$ is SSP.
\end{thm:productSSP}
\begin{proof}
	Let $F\subseteq K\times L$. We want to show that $|\Str(F)| \geq |F|$. For $k\in K$ and $\ell\in L$, define $F_k, I_k \subseteq L$ and $G_{\ell}, J_{\ell} \subseteq K$ as follows:
	
	\begin{align*}
	F_k &= \{\ell\in L \mid (k,\ell)\in F\}; \\
	I_k &= \Str_L(F_k); \\
	G_{\ell} &= \{k\in K \mid \ell \in I_k\}; \\
	J_{\ell} &= \Str_K(G_{\ell}).
	\end{align*}   	
	Construction of these sets is illustrated in Figure~\ref{fig-FIGJ} below.
	
	\begin{figure}[hbt]
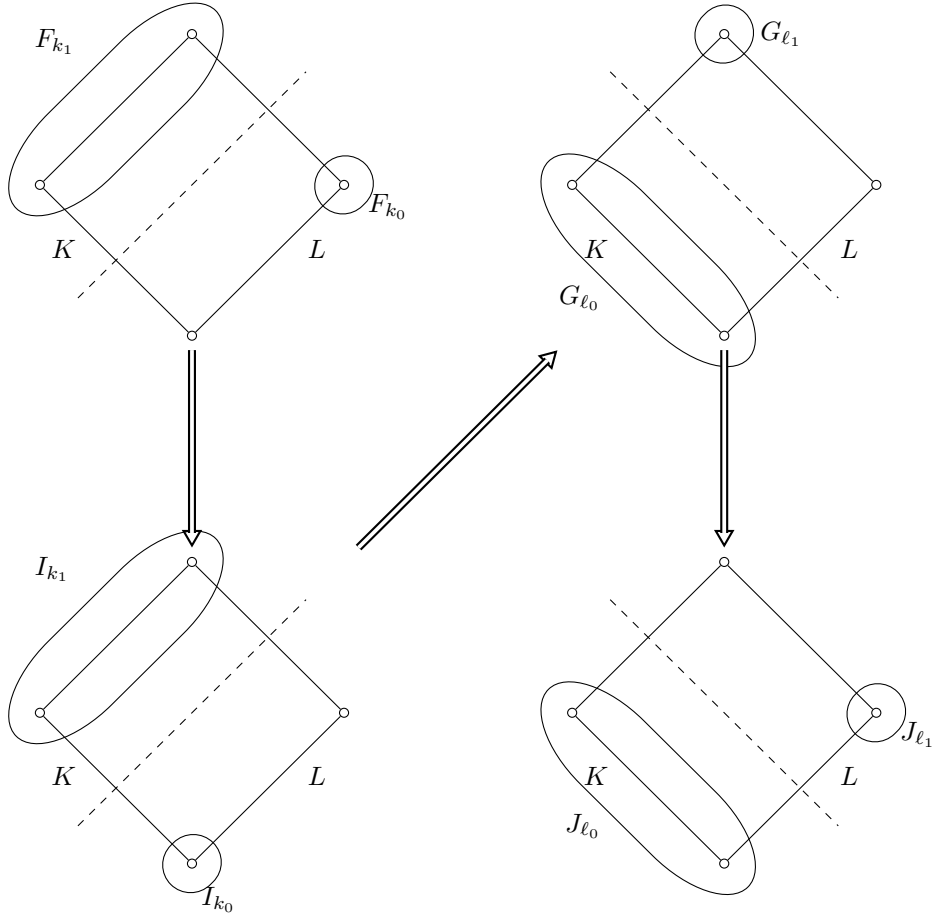

		\centering
		\include{FIGJ2}
		\caption{Construction of $F_k$, $I_k$, $G_{\ell}$ and $J_{\ell}$.}
		\label{fig-FIGJ}       
	\end{figure}

	Now, by SSP of $K$ and $L$, $|I_k|\geq |F_k|$ and $|J_{\ell}|\geq |G_{\ell}|$, for all $k\in K$, $\ell\in L$. Also, by definition, $\bigsqcup_{k\in K} \{k\} \times I_k = \bigsqcup_{\ell\in L} G_\ell \times \{\ell\}$. 
	Let us define $J = \bigsqcup_{\ell\in L} J_\ell \times\{\ell\}$. Then
	\[ |J| = \sum_{\ell\in L} |J_{\ell}| \geq \sum_{\ell \in L} |G_\ell| = \sum_{k\in K} |I_k| \geq \sum_{k\in K} |F_k| = |F|.\]
	To conclude the proof, we show that $F$ shatters all elements of $J$.
	
	Take any $(k,\ell) \in J$, and let $k' \leq k$ and $\ell' \leq \ell$. By construction, $k \in J_\ell$, and so $k$ is shattered by $G_\ell$. In particular, some $u \in G_\ell$ satisfies $k \meet_K u = k'$. Since $u \in G_\ell$, by construction $\ell \in I_u$, that is, $\ell$ is shattered by $F_u$. In particular, some $v \in F_u$ satisfies $\ell \meet_L v = \ell'$. In total, $(u,v) \in F$ satisfies $(k,\ell) \meet (u,v) = (k',\ell')$. 
%
\end{proof}


\subsection{Proof of Theorem~\ref{thm:n1SSP}}

We close by identifying a class of families which satisfies the SSP property in any RC lattice.

\newtheorem*{thm:n1SSP}{Theorem \ref{thm:n1SSP}}
\begin{thm:n1SSP}
	Let $L$ be an RC lattice and $F \subset L$ be a family for which $ L \backslash \Str(F)$ contains exactly one minimal element, then $\lvert F \rvert  \leq \lvert \Str(F)\rvert$. 
\end{thm:n1SSP}

\begin{proof}
	Denote the set of elements not shattered by $F$ by $N = L \backslash \Str(F)$. The set $N$ is closed upwards: if $F$ does not shatter $u$ then it also does not shatter any $v \geq u$. Therefore, if $x$ is the unique minimal element of $N$, then $N = [x) = \{u \mid x\leq u\}$.
	
	Since $x$ is not shattered by $F$, there exists some $y\leq x$ such that no $u \in F$ satisfies $u \meet x = y$. Denote $D = \{u \mid x\wedge u = y\}$, so $D$ and $F$ are disjoint. We will show that $|N| \leq |D|$. This implies the lemma since
	\[|F| = |L| - |L \backslash F| \leq |L| - |D| \leq |L| - |N| = |L| - |L \backslash \Str(F)| = |\Str(F)|.\]
	  
	Now, to prove $|N| \leq |D|$, let us take an arbitrary $a\in N = [x)$, that is, an arbitrary $a$ satisfying $x\leq a$. Then $y\leq x\leq a$, but, as $L$ is RC, the interval $[y,a]$ is complemented, and we can pick a complement $c(a)$ of $x$ in $[y,a]$. Note that this applies, in particular, when $y=x=a$, in which case $c(a) = a$.
	
	By definition, $c(a) \meet x = y$, that is, $c(a) \in D$. Also, $c(a) \join x = a$. This implies that the map $a \mapsto c(a)$ is one-to-one on $N$. Indeed, if $c(a) = c(b)$ then 
	$a = x\join c(a) = x\join c(b) = b$. The existence of a one-to-one mapping from $N$ to $D$ proves that $|N|\leq |D|$, finishing the argument. 
\end{proof}

\section{Conclusion and open problems} \label{sec:thoughts}

VC dimension and shattering were extensively studied in the classical case. \Cref{thm:main-intro-2} and, if true, \Cref{conj:SSP=RC}, enable us to ask related questions in an extended setting of lattices with nonvanishing M\"obius function, or of RC lattices. We end the paper by outlining a list of possible questions, which we consider to be interesting.

\emph{Families with small VC~dimension and inclusion-maximality.} An explicit description of all sets of VC~dimension~$1$ exists in the classical case: they correspond to forests~\cite{MR13,BenDavid15}. Having an exhaustive description of these families in the extended setup would be desirable. 
A subset of an SSP-lattice of VC~dimension~$k$ is called \emph{inclusion-maximal} if adding an element to it increases its VC~dimension. Naturally, study of families with small VC~dimension boils down to study of inclusion-maximal families of small VC dimension.

\begin{question}
	Give a description of inclusion-maximal families of VC~dimension~$1$ for SSP lattices. 
\end{question}

It is known that, in a classical setup, every set of VC~dimension~$1$ can be extended to a set of size $n+1$ without increasing its VC~dimension. The same does not hold in the case of vector spaces, as the following example shows:

\begin{lemma} \label{lem:critical-1}
	Consider a subspace lattice $\lattice = \FF_q^n$, where $n \geq 2$. Let $U$ be a subspace of $\FF_q^n$ of dimension~$n-1$. The set $F = \{ 0, \FF_q^n \} \cup \{ \Span{x} : x \notin U \}$ is an inclusion-maximal set of VC~dimension~$1$, and it contains $q^{n-1} + 2$ subspaces; in comparison, $\sqbinom{\lattice}{\leq 1} = 1 + \frac{q^n-1}{q-1}$, which is larger by a factor of roughly $\frac{q}{q-1}$.
\end{lemma}

\begin{question}
	Is there a ``nice'' characterization of lattices for which all inclusion-maximal sets of VC~dimension~$1$ have at least $n+1$ elements? Here $n$ is the number of join-irreducible elements, which in the classical setup corresponds to the size of the base set. 
\end{question}

\emph{Shattering-extremality.} 
A subset of an SSP-lattice  is called \emph{shattering-extremal} if it shatters as many elements as it has, that is, if it obtains equality in SSP inequality. A number of results exists for shattering-extremal families in the classical setup, in particular, for small VC~dimension~\cite{Moran12,MR13,MR14}. 

\begin{question}
	What can we say about shattering-extremal sets of small VC~dimension for SSP lattices?
\end{question}	

In the classical setting there is a notion of \emph{strong shattering}, which is in some sense dual to shattering. The analog of the SSP lemma for strong shattering states that a family strongly shatters at most as many elements as it has. Thus the size of the family is sandwiched between the number of elements it strongly shatters and the number of elements it shatters~\cite{BR95}. It turns out that equality holds in the SSP lemma iff it holds in the strong SSP lemma.

\begin{question}
	Can strong shattering be reasonably defined in the lattice setting? Will shattering-extremality be equivalent to strong shattering-extremality, as in the classical setting?
\end{question}	

Shattering-extremal families which are closed under intersection are precisely \emph{convex geometries}~\cite{Chorn18}.

\begin{question}
	Characterize the class of shattering-extremal meet-subsemilattices of SSP lattices.
	Matroids and their duals, being SSP, are in this class, as are duals of antimatroids, which are convex geometries.
	
	Matroids and antimatroids are known examples of \emph{greedoids}~\cite{Greedoids}. Are shattering-extremal families precisely duals of greedoids, or is there some other connection?
\end{question}	

It is well-known that for antimatroids there is a characterization in terms of \emph{forbidden projections}, also called \emph{circuits}~\cite{Dietrich87}. This characterization can be recast, in a straightforward manner, for convex geometries. A similar characterization also exists for shattering-extremal families in general~\cite{Forbidden}. 

\begin{question}
	Is there a ``forbidden projections'' characterization of shattering-extremal families of SSP lattices?
\end{question}

\bibliographystyle{plain}
\bibliography{biblio}

\end{document}

%% file: FIGJ2.tex
\begin{tikzpicture} 
[point/.style={inner sep = 1.2pt, circle,draw,fill=white},
FIT/.style args = {#1}{rounded rectangle, draw,  fit=#1, rotate fit=45, yscale=0.5},
FITR/.style args = {#1}{rounded rectangle, draw,  fit=#1, rotate fit=-45, yscale=0.5},
FIT1/.style args = {#1}{rounded rectangle, draw,  fit=#1, rotate fit=45, scale=2},
vecArrow/.style={
		thick, decoration={markings,mark=at position
		   1 with {\arrow[thick]{open triangle 60}}},
		   double distance=1.4pt, shorten >= 5.5pt,
		   preaction = {decorate},
		   postaction = {draw,line width=0.4pt, white,shorten >= 4.5pt}
	}
]

\begin{scope} [scale=0.5, xshift = -6cm]      
        \begin{scope} [yshift = 0cm]			  
        	\node(1) at (-4, 4) [point] {};
        	\node(2) at (4, 4) [point] {};
        	\node(12) at (0, 8) [point] {};
        	\node(b) at (0, 0) [point] {};
        	\node [below left] at ($(b)!0.7!(1)$){$K$};	
        	\node [below right] at ($(b)!0.7!(2)$){$L$};	
        	
        	\draw (b)--(1)--(12)--(2)--(b);
        	\draw [dashed] (-3,1)--(3,7);
        	
        	\node (F) [fit=(b)(1)(2)(12)] {};
        	\node [below right] at (F.north west) {$F_{k_1}$};
        	\node [below right] at (F.east) {$F_{k_0}$};
            \node[FIT=(1)(12)] {};
            \node[FIT1=(2)] {};
        \end{scope}	

        \begin{scope} [yshift = -14cm]			  
			\node(1) at (-4, 4) [point] {};
			\node(2) at (4, 4) [point] {};
			\node(12) at (0, 8) [point] {};
			\node(b) at (0, 0) [point] {};
			\node [below left] at ($(b)!0.7!(1)$){$K$};	
			\node [below right] at ($(b)!0.7!(2)$){$L$};	
			
			\draw (b)--(1)--(12)--(2)--(b);
			\draw [dashed] (-3,1)--(3,7);
			
			\node (I) [fit=(b)(1)(2)(12)] {};
			\node [below right] at (I.north west) {$I_{k_1}$};
			\node [below right] at (I.south) {$I_{k_0}$};
			
			\node[FIT=(1)(12)] {};
			\node[FIT1=(b)] {};
			
		\end{scope}	

        \begin{scope} [xshift=14cm, yshift = 0cm]			  
			\node(1) at (-4, 4) [point] {};
			\node(2) at (4, 4) [point] {};
			\node(12) at (0, 8) [point] {};
			\node(b) at (0, 0) [point] {};
			\node [below left] at ($(b)!0.7!(1)$){$K$};	
			\node [below right] at ($(b)!0.7!(2)$){$L$};	
			
			\draw (b)--(1)--(12)--(2)--(b);
			\draw [dashed] (-3,7)--(3,1);
			
			\node (G) [fit=(b)(1)(2)(12)] {};
			\node [left] at (-3,1) {$G_{\ell_0}$};
			\node [] at (1.5,8) {$G_{\ell_1}$};
			\node[FITR=(1)(b)] {};
			\node[FIT1=(12)] {};
		\end{scope}	

        \begin{scope} [xshift=14cm, yshift = -14cm]			  
			\node(1) at (-4, 4) [point] {};
			\node(2) at (4, 4) [point] {};
			\node(12) at (0, 8) [point] {};
			\node(b) at (0, 0) [point] {};
			\node [below left] at ($(b)!0.7!(1)$){$K$};	
			\node [below right] at ($(b)!0.7!(2)$){$L$};	
			
			\draw (b)--(1)--(12)--(2)--(b);
			\draw [dashed] (-3,7)--(3,1);
			
			\node (J) [fit=(b)(1)(2)(12)] {};
			\node [left] at (-3,1) {$J_{\ell_0}$};
			\node [below right] at (J.east) {$J_{\ell_1}$};
			\node[FITR=(1)(b)] {};
			\node[FIT1=(2)] {};
		\end{scope}	

        \draw[vecArrow] (F.south)--(I.north);
        \draw[vecArrow] (I.north east)--(G.south west);
        \draw[vecArrow] (G.south)--(J.north);
\end{scope}	

\end{tikzpicture}	